\documentclass[a4paper,11pt]{amsart}
\usepackage{hyperref,latexsym}
\usepackage{enumerate}

\theoremstyle{plain}
\newtheorem{theorem}{Theorem}[section]
\newtheorem{lemma}[theorem]{Lemma}
\newtheorem{corollary}[theorem]{Corollary}
\newtheorem{proposition}[theorem]{Proposition}
\theoremstyle{definition}
\newtheorem{definition}[theorem]{Definition}
\newtheorem{example}{Example}
\theoremstyle{remark}
\newtheorem{remark}{Remark}

\newcommand{\M}{\mathcal{M}}

\begin{document}

\title[Integrable Magnetic billiards]
      {Algebraic non-integrability of magnetic billiards}

\date{April 2016}
\author{Misha Bialy and Andrey E. Mironov}
\address{M. Bialy, School of Mathematical Sciences, Tel Aviv
University, Israel} \email{bialy@post.tau.ac.il}
\address{A.E. Mironov, Sobolev Institute of Mathematics, Novosibirsk, Russia } \email{mironov@math.nsc.ru}
\thanks{M.B. was supported in part by Israel Science Foundation grant
162/15 and A.M. was supported by Supported by RSF grant 14-11-0044.}


\begin{abstract} We consider billiard ball motion in
a convex domain of the Euclidean plane bounded by a piece-wise
smooth curve influenced by the constant magnetic field. We show that
if there exists a polynomial in velocities integral of the magnetic
billiard flow then every smooth piece $\gamma$ of the boundary must
be algebraic and either is a circle or satisfies very strong
restrictions. In particular in the case of ellipse it follows that
magnetic billiard is algebraically not integrable for all magnitudes
of the magnetic field. We conjecture that circle is the only
integrable magnetic billiard not only in the algebraic sense, but
for a broader meaning of integrability. We also introduce the model
of Outer magnetic billiards. As an application of our method we
prove analogous results on algebraically integrable Outer magnetic
billiards.
\end{abstract}

\maketitle



\section{\bf Introduction and main results}

\subsection{Birkhoff magnetic billiards}In this paper we consider
magnetic billiard inside a convex domain $\Omega$ bounded by simple
piece-wise smooth closed curve. We consider the influence of the
magnetic field of constant magnitude $\beta>0$ on the billiard
motion, so that the particle moves inside $\Omega$ with unite speed
along Larmor circle of constant radius $r=\frac{1}{\beta}$ in a
counterclockwise direction. Hitting the boundary the particle is
reflected according to the law of geometric optics. We call such a
model -- Birkhoff magnetic billiard.

We shall assume that every smooth piece $\gamma$ of the boundary of
$\Omega$ satisfies
$$
\beta< \min_{\gamma} k,
$$
where $k$ is the curvature. In other words we assume that the
magnetic field is relatively week with respect to the curvature. It
is an exercise in differential geometry of curves that in this case
the boundary of the domain $\Omega$ is strictly convex with respect
to circles of radius $r=\frac{1}{\beta}.$ This  means in particular
that intersection of any circle of radius $r$ with the domain
consists of at most one arc. Moreover, under this assumption, if a
circle of radius $r$ oriented in the same direction as the boundary
is tangent to $\partial\Omega$ (with agreed orientation) then it
contains the domain $\Omega$ inside. Birkhoff magnetic billiards
were studied in many papers:  \cite{Berg}, \cite{B}, \cite{GB},
\cite{k}, \cite{kozlova}, \cite{BR}, \cite{T2}, \cite{Ta}.

 Motivation
for the present paper comes from \cite{BR} where computer evidence
of chaotic behavior of magnetic billiard inside ellipse is
demonstrated for all magnitude of magnetic fields. For $\beta$
positive, unlike the case $\beta=0$, the pictures show that the
billiard is not integrable. We examine this problem in algebraic
setting using the ideas from our recent papers on usual Birkhoff
billiards \cite{BM1}, \cite{BM2} extending previous results of
\cite{Bolotin} and \cite{Tab}.

It seems plausible that other approaches for integrability might be
applicable. For example, it is interesting to test meromorphic
integrals via variational equations along 2-periodic orbit (see Fig. 1)
in the spirit of \cite{andrey}.

\vskip30mm

\begin{picture}(170,100)(-100,-50)

\qbezier(40,50)(44,83)(95,87)
\qbezier(95,87)(146,83)(150,50)
\qbezier(150,50)(146,18)(95,14)
\qbezier(95,14)(44,18)(40,50)

\put(95,87){\vector(-3,0){1}}
\put(147,25){\shortstack{$\gamma$}}
\qbezier(40,50)(95,71)(150,50)
\qbezier(40,50)(95,30)(150,50)
\put(94,61){\vector(-3,0){1}}
\put(97,40){\vector(3,0){1}}




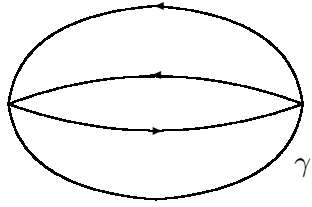
\put(15,-5){\shortstack{Fig. 1 {\footnotesize 2-periodic orbit inside $\Omega$, $\partial\Omega=\gamma.$}}}



\end{picture}

\subsection{Polynomial Integrals}

We shall study the existence of Polynomial in velocities integral
for magnetic billiards.
\begin{definition}\label{def} Let $\Phi:T_1\Omega\rightarrow\mathbb{R}$
be a function on the unite tangent bundle, $\Phi=\sum_{k+l=0}^N
a_{kl}(x)v_1^kv_2^l$, which is a polynomial in components of $v$
 with
continuous coefficients up to the boundary, $a_{kl}\in
C(\overline\Omega).$ We call $\Phi$ a polynomial integral of the
magnetic billiard if the following conditions hold.

1. $\Phi$ is an integral of magnetic flow $g^t$ inside $\Omega$,
$$\Phi(g^t(x,v)))=\Phi(x,v);$$

2. $\Phi$ is preserved under the reflections at smooth points of the
boundary: For any smooth point $x\in\partial\Omega,$
$$\Phi(x,v)=\Phi(x,v-2<n,v>n),$$ for any $v\in T_x\Omega, |v|=1,$
here $(v,n)$ positive orthonormal basis.
\end{definition}
\begin{remark}It appears that the condition of
convexity with respect to the circles of radius $r$ which we
introduced above can be relaxed if one adds in the Definition
\ref{def} of the integral the following requirement: For any given
circle of radius $r$ intersecting the domain $\Omega$ in several
arcs, the integral $\Phi$ is required to have the same value on all
the arcs of the intersection with $\Omega$. Our method strongly
relies on this additional requirement and it is not clear to us if
it is really necessary.
\end{remark}

\begin{example}\label{example}Let $\gamma$ be a circle with the center at the
origin. Then the function which measures the distance to the origin
of the center of Larmor circle remains unchanged under the
reflections and hence is the integral of the billiard flow. So the
integral $h$ in this case has the form:
$$h(x,v)=x_1^2+x_2^2+\frac{2}{\beta}(v_1x_2-v_2x_1).$$
\end{example}

In fact, we are not aware of any other piece-wise smooth example of
integrable magnetic billiard. Similarly to Birkhoff' conjecture for
usual billiard we ask if the only integrable magnetic billiard is
circular. As usual the integrability can be understood in various
ways. In this paper we study Polynomial in velocities integrals for
magnetic billiards. Another approach of \emph{Total} integrability
was considered in \cite{B}.

\subsection{Phase Space of Magnetic billiard}
We shall use throughout this paper the following construction.
Denote by $J$ the standard complex structure on $\mathbb{R}^2$ and
introduce the mapping:
\begin{equation}\label{L}{\mathcal L}: T_1\Omega \rightarrow \mathbb{R}^2,\
 {\mathcal L}(x,v)=x+rJv,
\end{equation}
which assigns to every unite tangent vector $v\in T_x{\Omega}$ the
center of the corresponding Larmor circle. Varying unite vector $v$
in $T_x{\Omega}$, for a fixed point $x\in\Omega$,  the corresponding
Larmor centers form a circle of radius $r$ centered at $x$, and the
domain swept by all these circles, when $x$ runs over $\Omega$, we
shall denote by $\Omega_r$:
$$\Omega_r={\mathcal L}(\Omega).$$
Vice versa, one can prove that for any circle of radius $r$ lying in
$\Omega_r$ its center necessarily belongs to $\Omega$.

For Birkhoff magnetic billiard we always choose a counterclockwise
orientation of $\partial\Omega.$
 Moreover, for any smooth piece $\gamma$ of the boundary
$\partial\Omega$ we define two curves as follows. Fix an arc-length
parameter $s$ of positive orientation, we set:
\begin{equation}
\label{D} \gamma_{+r}(s)={\mathcal L}(\gamma(s), \tau(s));\
\gamma_{-r}(s)={\mathcal L}(\gamma(s), -\tau(s)),
\end{equation}
where $\tau(s)=\dot{\gamma}(s).$ It is easy to see that,
$\Omega_r\subset\mathbb{R}^2$ is a bounded domain in the plane
homeomorphic to the annulus and the curves $\gamma_{\pm r}$, called
parallel curves to $\gamma,$ lie on the boundary
$\partial{\Omega_r}.$ Here $\gamma_{-r}$ lies on the outer boundary
of the annulus, and $\gamma_{+r}$ lies on the inner boundary. Other
pieces of the boundary of $\partial{\Omega_r}$ are circular arcs of
radius $r$ with the centers at the corners of the boundary
$\partial\Omega,$ but they will not play any role for us in the
sequel.
\begin{remark}
The curves $\gamma_{\pm r}$ are also called equidistant curves, or
fronts, in Singularity theory, or offset curves in Computer Aided
Geometric Design see \cite{SS}, \cite{SeS}.
\end{remark}
One easily computes the curvature of the parallel curves to be
$$k_{+r}=\frac{k(\gamma)}{r\cdot k(\gamma)-1};\quad
k_{-r}=\frac{k(\gamma)}{r\cdot k(\gamma)+1}. $$ So the curvature of
the inner boundary $k_{+r}$ and the outer boundary $k_{-r}$ always
satisfy the bounds:
\begin{equation}\label{bounds}
k_{+r}>\beta,\quad 0<k_{-r}<\beta,
\end{equation}
showing in particular, that any circle of radius $r$ with the center
at $\gamma(s)$ on $\gamma$ is tangent to the outer boundary from
inside at $\gamma(s)-rJ\dot{\gamma}(s)$ and to the inner boundary
from outside at the point $\gamma(s)+rJ\dot{\gamma}(s).$ Moreover,
apart from these tangencies this circle remains entirely inside
$\Omega_r$ (see Fig. 2).

By the definition ${\mathcal L}$ has constant value on every Larmor circle so
the components of ${\mathcal L}$ are integrals of magnetic flow  $g^t$ inside
$\Omega$.

Moreover, we introduce the mapping $\M:\Omega_r\rightarrow\Omega_r$
by the following rule: Let $C_-,C_+$ are two Larmor circles centered
at $P_-, P_+$ respectively. We define
$$
\M(P_-)=P_+ \iff\ C_-\quad  \textrm{is\ transformed \ to}\quad C_+,
$$
after billiard reflection at the boundary $\partial\Omega.$

With this definition $\M:\Omega_r\rightarrow\Omega_r$ preserves the
standard symplectic form in the plane, and thus $\Omega_r$ naturally
becomes the phase space of magnetic billiard. We shall call $\M$
Magnetic billiard map. Notice that on the boundaries $\gamma_{\pm
r}$, map $\M$ acts identically, while on the connecting circles of
the boundary corresponding to the corners of $\Omega$, map $\M$ is
not defined.

\vskip30mm

\begin{picture}(170,100)(-100,-50)

\qbezier(50,50)(51,69)(70,70)
\qbezier(70,70)(89,69)(90,50)
\qbezier(90,50)(89,31)(70,30)
\qbezier(70,30)(51,31)(50,50)

\put(50,50){\vector(0,-3){1}}

\qbezier(40,50)(41,79)(70,80)
\qbezier(70,80)(99,79)(100,50)
\qbezier(100,50)(99,21)(70,20)
\qbezier(70,20)(41,21)(40,50)

\put(40,50){\vector(0,-3){1}}

\qbezier(0,50)(1,119)(70,120)
\qbezier(70,120)(139,119)(140,50)
\qbezier(140,50)(139,-19)(70,-20)
\qbezier(70,-20)(1,-19)(0,50)

\put(0,50){\vector(0,-3){1}}

\qbezier(20,30)(22,78)(70,80)
\qbezier(70,80)(118,78)(120,30)
\qbezier(120,30)(118,-18)(70,-20)
\qbezier(70,-20)(22,-18)(20,30)

\put(20,30){\vector(0,-3){1}}

\put(67,51){\shortstack{$\Omega$}}
\put(67,97){\shortstack{$\Omega_r$}}

\put(130,100){\shortstack{$\gamma_{-r}$}}
\put(33,23){\shortstack{$\gamma_{+r}$}}

\put(70,30){\circle*{2}}

\put(62,35){\shortstack{$\gamma(s)$}}

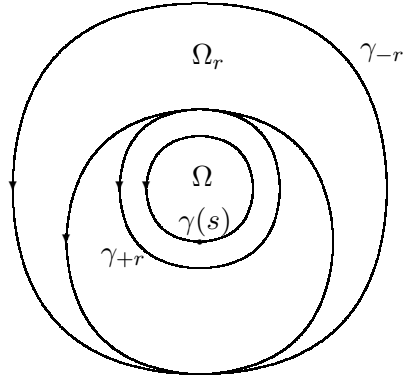
\put(-77,-35){\shortstack{Fig. 2 {\footnotesize Circle of radius $r$
centered at $\gamma(s)$ is tangent to $\partial \Omega_r=
\gamma_{+r}\cup\gamma_{-r}.$}}}

\end{picture}

Given a polynomial integral
$\Phi=\sum_{ k+l=0}^N
a_{kl}(x)v_1^kv_2^l$ of the magnetic billiard we define the function
$F:\Omega_r\rightarrow\Omega_r$ by the requirement:
\begin{equation}\label{F}
F\circ{\mathcal L}=\Phi.
\end{equation}
 This is a well defined construction since $\Phi$ is an
integral of the magnetic flow so has constant values on any Larmor
circle. Moreover, since $\Phi$ is invariant under the billiard flow,
it follows that $F$ is invariant under billiard map $\M$:
$$
F\circ\M=F.
$$

 In
coordinates the definition (\ref{F}) reads:
\begin{equation}\label{v} F(x_1-rv_2,x_2+rv_1)=\Phi(x,v)=\sum_{0
\leq k+l\leq N}a_{kl}(x)v_1^kv_2^l.
\end{equation}
Notice, that since $\Phi$ is a polynomial in $v$, then function $F$
satisfies the following  property: $F$ restricted to any circle of
radius $r$ lying in $\Omega_r$ is a trigonometric polynomial of
degree at most $N$. Next Theorem claims that in such a case $F$ is a
polynomial function:
\begin{theorem}\label{harmonics} Let $\Omega_r$ be a domain in
$\mathbb{R}^2$ which is the union of all circles of radius $r$ whose
centers run over a domain $\Omega$ (for example the whole
$\mathbb{R}^2$). Let $F:\Omega_r\rightarrow\mathbb{R}$ be a
continuous function on $\Omega_r$ such that $F$ being restricted to
any circle of radius $r$ of $\Omega_r$ is a trigonometric polynomial
of degree at most $N.$ It then follows that $F$ is a polynomial in
$x,y$ of degree at most $2N.$
\end{theorem}
Proof of this theorem uses Lemma \ref{lemma} whose elegant proof was
communicated to us by S.Tabachnikov. Proofs of the lemma and Theorem
\ref{harmonics} are given in Section \ref{lemma}. Notice that if one
allows $r$ to be arbitrary in Theorem \ref{harmonics} then the fact
would be obvious, however the result holds true when $r$ is fixed.

The next Corollary immediately follows from the Theorem and the
relation (\ref{v}):
\begin{corollary}
 Coefficients of the integral $\Phi$ are polynomials in $x,y$ of
degree at most $2N-(k+l)$.
\end{corollary}

Recall that the coefficients of integral $\Phi$ are assumed to be
continuous on the closure of $\Omega$, therefore it follows from the
very construction of the function $F$ that it extends continuously
to the boundary $\partial{\Omega_r}$ and hence by Theorem
\ref{harmonics}, $F$ coincides with a polynomial on
$\overline{\Omega}_r$.

Moreover, we will prove the following:
\begin{proposition}\label{prop}
Suppose that the magnetic billiard in $\Omega$ admits a polynomial
integral $\Phi$ and let $F$ be the corresponding polynomial on
$\overline{\Omega}_r$. Then for every smooth piece $\gamma$ of the
boundary $\partial\Omega$ it follows that
$$F|_{\gamma_{\pm r}}=const.$$
\end{proposition}
\begin{remark}\label{r}
Given a smooth piece $\gamma$ of the boundary one can assume that
polynomial integral $F$ of $\M$ is such that the $constant$ in the
Proposition is $0,$ for both parallel curves $\gamma_{\pm r}$. Indeed if
$F|_{\gamma_{- r}}=c_1$ and $F|_{\gamma_{+ r}}=c_2,$ one can replace $F$
by $F^2-(c_1+c_2)F+c_1\cdot c_2$ to annihilate both constants
$c_1,c_2.$ In terms of the integral $\Phi$ this means that on $\gamma$
one can assume that:
$$\Phi(x,\pm\tau)=0,$$ for every point $x\in\gamma$ and $\tau$
a unite tangent vector to $\gamma$ at $x$.
\end{remark}
We shall prove Proposition \ref{prop} in Section 2

Proposition \ref{prop} and the Remark \ref{r} imply:
$$\gamma_{\pm r}\subset \{F=0\},$$
and thus $\gamma_{\pm r}$ is contained in the algebraic curve
$\{F=0\}$. This fact implies then that $\gamma$ itself is also
algebraic. We shall denote in the sequel the minimal defining
polynomial of the irreducible component in $\mathbb{C}^2$ containing
$\gamma_{\pm r}$ by $f_{\pm r}$ respectively. Since the curves
$\gamma_{\pm r}$ are real $f_{\pm r}$ are with real coefficients.
Notice, that it may happen that both $\gamma_{\pm r}$ belong to the
same component, so that $f_{+r}=f_{-r}.$ For instance this happens
for the case of parallel curves to $\gamma$ when $\gamma$ is ellipse
\cite {FN}, \cite{SS}, \cite{SeS}. In this case $f_{-r}=f_{+r}$ is
the irreducible polynomial of degree $8$.

\subsection{Main result and Corollaries}
We turn now to the formulation of our main result:
\begin{theorem}\label{main}
Let $\Omega$ be a convex bounded domain with a piece-wise smooth
boundary, such that every smooth piece of the boundary has curvature
at least, $\beta$. Suppose that the magnetic billiard in $\Omega$
admits a Polynomial integral $\Phi.$ Then the following alternative
holds: either $\partial\Omega$ is a circle, or every smooth piece
$\gamma$ of the boundary $\partial\Omega$ is not circular and has
the property that affine curves $\{f_{\pm r}=0\}$ are smooth in
$\mathbb{C}^2.$ Moreover, any non-singular point of intersection of
the projective curve $\{\tilde{f}_{\pm r}=0\}$ in $\mathbb{C}P^2$
with the infinite line $\{z=0\}$ away from isotropic points $(1:\pm
i:0)$ must be a tangency point with the infinite line. Here
$\tilde{f}_{\pm r}$ is a homogenization of $f_{\pm r}$.
\end{theorem}
\begin{corollary}\label{all}
For any non-circular domain $\Omega$ in the plane, the magnetic
billiard inside $\Omega$ is not algebraically integrable for all but
finitely many values of $\beta$.
\end{corollary}
\begin{proof}
Indeed, $f_{\pm r}$ depends on $r$ as polynomial function, so
$f_{\pm r}$ is a polynomial in $x,y$ and $r$.
 Moreover, since every piece
$\gamma$ has positive curvature bounded from below by $\beta$, then
there is an open interval $\frac{1}{r}\in(k_{\min};k_{\max})$ where
one can claim using differential geometry argument that
$\gamma_{+r}$ does have singularities. Hence, the system of
equations
$$
 \partial_x f_{+r}=\partial_y f_{+r}=f_{+r}=0
$$
defines in ${\mathbb C}^3$ an algebraic curve and its projection on
 $r$-line is Zariski open set.
 It then follows that singularities persist for all
but finitely many $r.$
\end{proof}

It may happen that refining our method below one can prove that the
only magnetic billiard admitting polynomial integral is circular.
But this is out of reach at present moment.

\begin{remark}
Our main result implies in particular that if one of the arcs of
magnetic billiard is circular then only the circle is algebraically
integrable. This is in contrast to the zero magnetic field case
where there exist polygons such that billiard flows admit polynomial
integrals (see \cite{KT}).
\end{remark}
\begin{corollary}Let $\Omega$ be an interior of the standard ellipse:
$$\gamma=\left\{\frac{x^2}{a^2}+\frac{y^2}{b^2}=1\right\},\quad 0<b<a.$$
 Then the for magnetic field of all magnitudes
 $0<\beta<k_{\min}=\frac{b}{a^2}$
magnetic billiard in ellipse is not algebraically integrable.
\end{corollary}

\begin{proof}
The equation of the parallel curves for ellipse is the following
(see for instance \cite{FN}):
$$
a^8 (b^4+(r^2-y^2)^2-2 b^2 (r^2+y^2))+b^4 (r^2-x^2)^2 (b^4-2 b^2 (r^2-x^2+y^2)+(x^2+y^2-r^2)^2)
$$
$$
-2
a^6 (b^6+(r^2-y^2)^2 (r^2+x^2-y^2)-b^4 (r^2-2 x^2+3 y^2)-b^2 (r^4+3 y^2 (x^2-y^2)+
$$
$$
r^2 (3
x^2+2 y^2)))+
2 a^2 b^2 (-b^6 (r^2+x^2)-(-r^2+x^2+y^2)^2 (r^4-x^2 y^2-r^2 (x^2+y^2))+
$$
$$
b^4 (r^4-3 x^4+3
x^2 y^2+r^2 (2 x^2+3 y^2))+b^2 (r^6-2 x^6+x^4 y^2-3 x^2 y^4+r^4(-4 x^2+2 y^2)+
$$
$$
r^2 (5 x^4-3 x^2 y^2-3 y^4)))+
a^4 (b^8+2 b^6 (r^2+3 x^2-2 y^2)+(r^2-y^2)^2 (-r^2+x^2+y^2)^2-
$$
$$
2 b^4 (3 r^4-3 x^4+5 x^2 y^2-3 y^4+4
r^2 (x^2+y^2))+2 b^2 (r^6-3 x^4 y^2+x^2 y^4-2 y^6+
$$
$$
 2 r^4 (x^2-2 y^2)+r^2 (-3 x^4-3 x^2 y^2+5 y^4)))=0.
$$
It turns out to be irreducible.
Moreover the parallel curves $\gamma_{\pm r}$ have singularities in
the Complex plane for every $r>\frac{1}{k_{\min}}=\frac{a^2}{b}>b$
as follows:
$$
 (0,\pm\frac{\sqrt{b^2-a^2}\sqrt{a^2-r^2}}{a}),\ (\pm\frac{\sqrt{a^2-b^2}\sqrt{b^2-r^2}}{b},0).
$$
Therefore the result follows from
Theorem \ref{main}.
\end{proof}

\subsection{Outer magnetic billiards}
It is remarkable that the action of billiard map $\M$ coincides
with, what we call Outer magnetic billiard. In addition, the result
which we get by our method provides the extension of Theorem by
Tabachnikov \cite{Tab} to the case of Outer magnetic billiards.

















 Let us introduce Outer magnetic billiard in
a natural way. Let $\Gamma$ be a smooth convex curve in the plane
with a fixed orientation (not necessarily counterclockwise). Let
$\beta>0$ be the magnitude of the magnetic field. Given a point $P$
outside $\Gamma$ we define $T(P)$ as follows: Consider the Larmor
circle of radius $r=\frac{1}{\beta}$ starting from $P$ tangent to
$\Gamma$ at $\Gamma(s)$ with the agreed orientation at $\Gamma(s)$
and then define $T(P)$ to lie on the same Larmor circle so that the
arcs $(P;\Gamma(s))$ and $(\Gamma(s);T(P))$ have the same angular
measure.

















Notice that there are two different cases:

1) In this case the orientation on $\Gamma$ is clockwise then $T$ is
well defined for any $\beta>0$ (see Fig. 3).

\vskip20mm

\begin{picture}(170,100)(-100,-60)

\qbezier(10,30)(11,49)(30,50)
\qbezier(30,50)(49,49)(50,30)
\qbezier(50,30)(49,11)(30,10)
\qbezier(30,10)(11,11)(10,30)

\put(30,50){\vector(3,0){1}}

\put(27,30){\shortstack{$\Gamma(s)$}}

\put(7,45){\shortstack{$\Gamma$}}

\qbezier(50,30)(52,78)(100,80)
\qbezier(100,80)(148,78)(150,30)
\qbezier(150,30)(148,-18)(100,-20)
\qbezier(100,-20)(52,-18)(50,30)

\put(71,72){\circle*{2}}
\put(71,-12){\circle*{2}}

\put(63,74){\shortstack{$P$}}
\put(51,-23){\shortstack{$T(P)$}}

\put(50,30){\circle*{2}}

\put(50,30){\line(0,-4){20}}
\put(50,10){\vector(0,-3){1}}
\put(43,2){\shortstack{$\tau$}}

\put(50,30){\line(4,0){20}}
\put(70,30){\vector(3,0){1}}
\put(70,34){\shortstack{$J\tau$}}

\put(152,20){\shortstack{$\Gamma_{_{+2r}}(s)$}}

\put(100,80){\vector(-3,0){1}}

\put(100,30){\circle*{2}}
\put(100,20){\shortstack{$O$}}


\put(150,30){\circle*{2}}

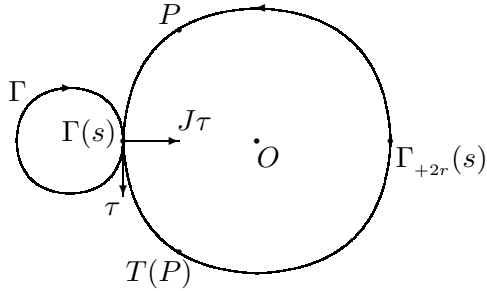
\put(-50,-50){\shortstack{Fig. 3 {\footnotesize Outer billiard map
$P\rightarrow T(P)$ for clockwise orientation on $\Gamma$.}}}

\end{picture}

2) However, if the orientation on $\Gamma$ is counterclockwise then
$T$ is well defined for $0<\beta<k_{min}$ (see Fig. 4).

\vskip25mm

\begin{picture}(170,100)(-100,-60)

\qbezier(50,60)(51,79)(70,80)
\qbezier(70,80)(89,79)(90,60)
\qbezier(90,60)(89,41)(70,40)
\qbezier(70,40)(51,41)(50,60)

\put(90,45){\shortstack{$\Gamma$}}

\put(71,85){\shortstack{$\Gamma(s)$}} \put(70,80){\circle*{2}}

\put(70,80){\line(-4,0){20}}
\put(50,80){\vector(-3,0){1}}
\put(45,85){\shortstack{$\tau$}}

\put(70,80){\line(0,-4){20}}
\put(70,60){\vector(0,-3){1}}
\put(72,55){\shortstack{$J\tau$}}

\qbezier(20,30)(22,78)(70,80)
\qbezier(70,80)(118,78)(120,30)
\qbezier(120,30)(118,-18)(70,-20)
\qbezier(70,-20)(22,-18)(20,30)

\put(108,65){\circle*{2}}
\put(110,67){\shortstack{$P$}}
\put(32,65){\circle*{2}}
\put(6,67){\shortstack{$T(P)$}}

\put(70,-20){\circle*{2}}
\put(57,-30){\shortstack{$\Gamma_{_{+2r}}(s)$}}

\put(20,30){\vector(0,-3){1}}
\put(50,60){\vector(0,-3){1}}

\put(70,30){\circle*{2}}
\put(60,20){\shortstack{$O$}}


\put(70,30){\circle*{2}}
\put(60,20){\shortstack{$O$}}

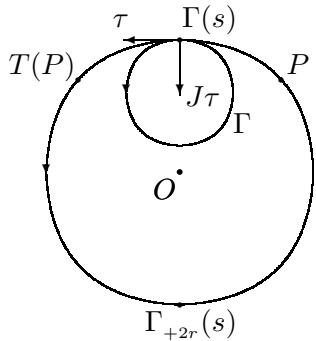
\put(-50,-50){\shortstack{Fig. 4 {\footnotesize Outer billiard map
$P\rightarrow T(P)$ for counterclockwise orientation on $\Gamma$.
}}}

\end{picture}

\begin{remark}
In fact in the second case one can allow that $\Gamma$ is
$C^1-$smooth
 curve which is piece-wise $C^2$, having arcs of radius
$r$ interlaced between non-circular $C^2$ pieces.
\end{remark}

 In
both cases 1) and2) the domain where the Outer billiard map $T$ is
defined is the annulus $A$ bounded by $\Gamma$ and $\Gamma_{+2r}$
(see Figure 3, 4). We call the dynamical system $T$ the Outer
magnetic billiard. It is not hard to check that $T$ is a symplectic
map of the Annulus $A$. Next theorem shows that Outer magnetic
billiard $T$ in the case 2) is in fact isomorphic to Birkhoff
magnetic billiard $\mathcal{M}$:
\begin{theorem}
Magnetic billiard map $\M:\Omega_r\rightarrow \Omega_r$ (defined in
subsection 1.3) coincides with Outer billiard $T$ determined by the
inner boundary $\gamma_{+r}$ of $\Omega_{r},$ endowed with a
counterclockwise orientation.
\end{theorem} Indeed,
let $P_-,P_+$ be the centers of two Larmor circles $C_-,C_+$ such
that $C_-$ is reflected to $C_+$ at the point $Q$ of the smooth part
$\gamma$ of the boundary $\partial\Omega$. Then it follows from the
definition of reflection law, that the circle of radius $r$ with the
center in $Q$ oriented counterclockwise passes from $P_-$ to $P_+$
and is tangent $\gamma_{+r}$ at the point $P={\mathcal
L}(Q,\tau(Q))$ (see Fig. 5, 6 for two different configurations of
$C_-$ and $C_+$).

\vskip20mm

\begin{picture}(170,100)(-100,-50)

\qbezier(50,50)(51,69)(70,70)
\qbezier(70,70)(89,69)(90,50)
\qbezier(90,50)(89,31)(70,30)
\qbezier(70,30)(51,31)(50,50)

\put(70,70){\vector(-3,0){1}}

\qbezier(50,50)(56,37)(70,30)
\qbezier(90,50)(84,37)(70,30)
\put(57,40){\vector(1,-1){1}}
\put(83,40){\vector(1,1){1}}

\put(50,50){\circle*{2}}
\put(90,50){\circle*{2}}
\put(43,51){\shortstack{$q$}}
\put(92,51){\shortstack{$p$}}

\put(100,72){\circle*{2}}
\put(40,72){\circle*{2}}
\put(102,73){\shortstack{$P_{-}$}}
\put(24,74){\shortstack{$P_{+}$}}

\qbezier(40,50)(41,79)(70,80)
\qbezier(70,80)(99,79)(100,50)
\qbezier(100,50)(99,21)(70,20)
\qbezier(70,20)(41,21)(40,50)

\put(40,50){\vector(0,-3){1}}

\put(70,80){\circle*{2}}
\put(67,83){\shortstack{$P$}}

\qbezier(20,30)(22,78)(70,80)
\qbezier(70,80)(118,78)(120,30)
\qbezier(120,30)(118,-18)(70,-20)
\qbezier(70,-20)(22,-18)(20,30)

\put(20,30){\vector(0,-3){1}}

\put(66,51){\shortstack{$\Omega$}}

\put(34,23){\shortstack{$\gamma_{_{+r}}$}}

\put(70,30){\circle*{2}}

\put(65,35){\shortstack{$Q$}}

\put(-39,-35){\shortstack{Fig. 5 {\footnotesize Arcs $(q;Q)$ and $(Q;p)$ belong to $C_{-}$ and $C_{+}$.}}}

\end{picture}

\vskip20mm

\begin{picture}(170,100)(-100,-55)

\qbezier(50,10)(51,29)(70,30)
\qbezier(70,30)(89,29)(90,10)
\qbezier(90,10)(89,-9)(70,-10)
\qbezier(70,-10)(51,-9)(50,10)

\put(70,-10){\vector(3,0){1}}

\qbezier(50,10)(64,19)(70,30)
\qbezier(90,10)(76,19)(70,30)

\put(58,16){\vector(1,1){1}}
\put(82,16){\vector(1,-1){1}}

\put(50,10){\circle*{2}}
\put(90,10){\circle*{2}}
\put(43,11){\shortstack{$q$}}
\put(92,11){\shortstack{$p$}}

\put(115,54){\circle*{2}}
\put(26,55){\circle*{2}}
\put(27,49){\shortstack{$P_{-}$}}
\put(102,49){\shortstack{$P_{+}$}}

\put(70,80){\circle*{2}}
\put(67,85){\shortstack{$P_0$}}

\qbezier(10,55)(30,79)(70,80)
\qbezier(70,80)(110,79)(130,55)
\put(125,64){\shortstack{$\gamma_{_{-r}}$}}

\put(16,61){\vector(-1,-1){1}}

\qbezier(40,10)(41,39)(70,40)
\qbezier(70,40)(99,39)(100,10)
\qbezier(100,10)(99,-19)(70,-20)
\qbezier(70,-20)(41,-19)(40,10)

\put(40,10){\vector(0,-3){1}}

\put(70,-20){\circle*{2}}
\put(66,-32){\shortstack{$P$}}

\qbezier(20,30)(22,78)(70,80)
\qbezier(70,80)(118,78)(120,30)
\qbezier(120,30)(118,-18)(70,-20)
\qbezier(70,-20)(22,-18)(20,30)

\put(20,30){\vector(0,-3){1}}

\put(66,1){\shortstack{$\Omega$}}

\put(45,44){\shortstack{$\gamma_{_{+r}}$}}

\put(70,30){\circle*{2}}

\put(65,16){\shortstack{$Q$}}

\put(-39,-50){\shortstack{Fig. 6 {\footnotesize Arcs $(q;Q)$ and $(Q;p)$ belong to $C_{-}$ and $C_{+}$.}}}

\end{picture}

\begin{remark}
 Let us
remark, that the map $T$ in the case 1) (see Figure 3) is not
isomorphic to Magnetic Birkhoff billiard globally, by topological
reasons. This can be seen already for the case when $\Gamma$ is a
circle. Indeed the difference between rotation numbers of two
boundaries $\Gamma,\Gamma_{+2r}$ equals 0 for case 1) and equals
$2\pi$ in case 2). Nevertheless, since our method below is
concentrated near the boundaries it applies for both cases 1) and
2).
\end{remark}

Let $F$ be a polynomial which is invariant under $T.$ As before we
shall call $F$ polynomial integral. Then similarly to the Birkhoff
magnetic billiard we have that $\Gamma,\Gamma_{+2r}$ lie in
$\{F=const\}$ and therefore are algebraic (similarly to Proposition
\ref{prop}). Our result for outer magnetic billiards reads:

\begin{theorem}
\label{outer} Assume that there exists a non-constant Polynomial $F$
such that $F$ is invariant under Outer billiard map $T$. Let
$f,f_{+2r}$ are irreducible defining polynomials of
$\Gamma,\Gamma_{+2r}.$ Then the following alternative holds. Either
$\Gamma$ is a circle, or the curves $\{f=0\},\{f_{+2r}=0\}$ in
$\mathbb{C}^2$ are smooth with the property that any non-singular
intersection point of the projective curves
$\{\tilde{f}=0\},\{\tilde{f}_{+2r}=0\}$ in $\mathbb{C}P^2$ (here
$\tilde{f}$ is a homogenization of $f$) with the infinite line
$\{z=0\}$ which is not an isotropic point $(1:\pm i:0),$ must be a
point of tangency.
\end{theorem}
\begin{corollary}
The outer magnetic billiard for ellipse is not algebraically
integrable.
\end{corollary}
\begin{proof}
For the case of ellipse the curve $\{\tilde{f}=0\}$ is smooth
everywhere in $\mathbb{C}P^2$ and intersects transversally the
infinite line in two points away from the isotropic points. Thus
Theorem \ref{outer} implies non-existence of polynomial integral.
\end{proof}
 Exactly as in Corollary \ref{all} we have the following:
\begin{corollary}\label{allouter} For all but finitely many
values of the magnitude of magnetic field $\beta$, the Outer
magnetic billiard of $\Gamma$ is not algebraically integrable unless
$\Gamma$ is a circle.
\end{corollary}
The rest of the paper is organized as follows. In Section \ref{two}
we prove Theorem \ref{harmonics} and Lemma \ref{lemma}. In Section
\ref{three} we deal with the boundary values of the integral $F$. In
Section \ref{remarkable} we derive a remarkable equation on $F$. In
Section \ref{mainproof} we prove Theorem \ref{main}. The proofs of
Theorem \ref{outer} and Corollary \ref{allouter} are completely
analogous and therefore are omitted.

\section*{Acknowledgements}The first named author is grateful to
Eugene Shustin and Misha Sodin for very useful consultations. We
thank Sergei Tabachnikov for communicating to us elegant proof of
Lemma \ref{lemma}.


\section{\bf Proof of Theorem \ref{harmonics}}\label{two}
We shall use the following lemma.
\begin{lemma}\label{lemma}
Let $F$ be a $C^{\infty}$ function $F:A\rightarrow \mathbb{R}$ where
$$A=\{(x,y): (r-\delta)^2\leq x^2+y^2\leq (r+\delta)^2\},$$
is the annulus in $\mathbb{R}^2$. Suppose that the function $F$
being restricted to any circle of radius $r$ lying in $A$ is a
trigonometric polynomial of degree at most $N.$ It then follows that
$F$ is a polynomial in $x$ and $y$ of degree at most $2N.$
\end{lemma}

\begin{proof} (after S. Tabachnikov). We shall say that $F$ has
property $P_N$ if the restriction of $F$ to any circle of radius $r$
lying in $A$ is a trigonometric polynomial of degree at most $N$.
The proof of Lemma goes by induction on the degree $N$.

1) For $N=0$, Lemma obviously holds since if $F$ has property $P_0$
then $F$ is a constant on any circle of radius $r$ and hence must be
a constant on the whole $A$, because any two points of $A$ can be
connected by a union of finite number circular arcs of radius $r$.

2) Assume now that any function satisfying property $P_{N-1}$  is a
polynomial of degree at most $2(N-1).$

Let $F$ be any smooth function on $A$ of property $P_N$. Denote by
$C_0$ be the core circle of $A$, i.e. $C_0=\{x^2+y^2=r^2\},$ and let
$F_0$ be the polynomial in $(x,y)$ of degree $N$ satisfying
$F|_{C_0}=F_0|_{C_0}.$ Then,
 one can find a $C^{\infty}$ function
$G:A\rightarrow\mathbb{R}$ so that
\begin{equation}\label{F0}
F(x,y)-F_0(x,y)=(x^2+y^2-r^2)G(x,y), \quad \forall (x,y)\in A.
\end{equation}
(This can be proven with the help of "polar" coordinates on $A$
$$(x,y)\rightarrow (u,v);\ u=x^2+y^2-r^2,\  v=\arg(x+iy),$$
applying Hadamard' lemma to the function $F-F_0$ with respect to the
variable $u$ and $v$ being the parameter.) Let us show now that $G$
has property $P_{N-1}$. Then by induction we will have that $G$ is a
polynomial of degree $2(N-1)$ and thus by (\ref{F0}), $F$ is a
polynomial of degree $2N$ at most. We need to show that the function
$g:=G|_C$ is a trigonometric polynomial of degree $(N-1)$ or less,
for any circle $C$ of radius $r$ in $A$ . With no loss of generality
we may assume that the circle $C$ is centered on the $x-$axes
(otherwise apply suitable rotation of the plane). Then
$$C=\{(x,y)\in A: (x-a)^2+y^2=r^2\}, \quad |a|<\delta.$$
Substituting $x=a+r\cos t, y=r\sin t$ into (\ref{F0}) we have
$$
(F-F_0)|_C=(a^2+2ar\cos t)\cdot g.
$$
 Writing the left and the right hand side
in Fourier series we get
$$
\sum_{-\infty}^{+\infty}f_ke^{ikt}=
a(a+re^{it}+re^{-it})\sum_{-\infty}^{+\infty}g_ke^{ikt},
$$ where $f_k$ are Fourier coefficients of $(F-F_0)|_C$.
Moreover, we have: $$f_k=0,\quad  |k|>N,$$ since both $F,F_0$ have
property $P_N.$ Thus we obtain linear recurrence relation for the
coefficients $g_k$:
$$
rg_{k+1}+ag_k+rg_{k-1}=0, \quad  |k|>N.
$$
The characteristic polynomial of this difference equation
$$\lambda^2+\frac{a}{r}\lambda+1=0$$ has two complex conjugate roots
$\lambda_{1,2}=e^{\pm i\alpha}$ and therefore we get the formula:
$$
g_{N+l}=c_1e^{il\alpha}+c_2e^{-il\alpha}, \quad  l\geq2,
$$ where
$$c_1+c_2=g_N,\quad c_1e^{i\alpha}+c_2e^{-i\alpha}=g_{N+1}.
$$
It is obvious now that if at least one of the coefficients $g_N$ or
$g_{N+1}$ does not vanish, then at least one of the constants $c_1,
c_2$ does not vanish and therefore the sequence $\{g_{N+l}\}$ does
not converge to $0$ when $l\rightarrow +\infty$. This contradicts
the continuity of $g$. Therefore both $g_N,g_{N+1}$ must vanish and
so $g$ is a trigonometric polynomial of degree at most $(N-1)$,
proving that $G$ has property $P_{N-1}$. This completes the proof.
\end{proof}

Next we give the proof of Theorem \ref{harmonics}.
\begin{proof}

Take any circle of radius $r$ lying in $\Omega_r$ and let $A$ be the
annulus which is the closure of its $\delta-$neighborhood.  Using
the convolution with a $C^{\infty}$ mollifier $\rho_{\epsilon}$
compactly supported in a small disc of radius $\epsilon$, we get a
$C^{\infty}$ function $F_{\epsilon}:$
$$
F_\epsilon(z):=\int \rho_{\epsilon}(z-\xi)F(\xi)d\xi=\int
F(z-\xi)\rho_{\epsilon}(\xi)d\xi,\quad z=(x,y).
$$
It is easy to see, that if $F$ has property $P_N$ then also
$F_{\epsilon}$ has property $P_N$ on the chosen annulus $A$ for all
$\epsilon$ small enough, $0<\epsilon<\epsilon_0$. Then by Lemma
\ref{lemma}, $F_{\epsilon}$ must be a polynomial on $A$ of degree at
most $2N$, for $0<\epsilon<\epsilon_0$. Recall, that $F_{\epsilon}$
converge to $F$ uniformly on $A$ as $\epsilon\rightarrow 0$.
Therefore, since the space of Polynomials of degree at most $2N$ is
finite-dimensional it then follows that $F$ is also a polynomial on
$A$ of degree at most $2N$. The set $\Omega_r$ can be covered by
annuli like $A$, therefore $F$ must be a polynomial of degree at
most $2N$ on the whole $\Omega_r.$ This completes the of Theorem
\ref{harmonics}.

\end{proof}

\section{\bf Boundary values of the Integral}
\label{three}We prove now Proposition \ref{prop}

\begin{proof} Take a point $Q$ on a smooth piece $\gamma$
of the boundary $\partial\Omega.$  Let $\tau(Q)$ be a positive unite
tangent vector to $\gamma$. Let $C_-,C_+$ be the incoming and
outgoing circles with the unite tangent vectors $v_-$ and $v_+$ at
the impact point $Q$. We are interested in the two cases when the
reflection angle between $\tau$ and $v_-$ or between $v_-$ and
$-\tau$ is close to zero. These two possibilities correspond (see
Fig. 5, 6) to the following cases:

$$(a)\quad\quad v_-=R_{-\epsilon}\tau,\quad\ \quad v_+=R_{\epsilon}\tau;$$
$$(b)\quad v_-=R_{\epsilon}(-\tau),\quad v_+=R_{-\epsilon}(-\tau),$$
where $R_{\epsilon}$ is the counterclockwise rotation of the plane
on a small angle $\epsilon$. On Figures 5,6  the arcs $(q;Q)$ and
$(Q;p)$ are the arcs of the circles $C_-$ and $C_+$ respectively.

We define
$$ P_-(\epsilon)={\mathcal L}(Q,v_-)=Q+rJ(v_-);
\quad P_+(\epsilon)={\mathcal L}(Q,v_+)=Q+rJ(v_+).
$$
In the case (a) we have
\begin{equation}\label{a}
P_-(\epsilon)=Q+rJ(R_{-\epsilon}\tau);\quad
P_+(\epsilon)=Q+rJ(R_{\epsilon}\tau). \end{equation}As for the case
(b): \begin{equation}\label{b}
P_-(\epsilon)=Q-rJ(R_{\epsilon}\tau);\quad
P_+(\epsilon)=Q-rJ(R_{-\epsilon}\tau).
\end{equation}
On the Figures 5, 6 we abbreviate

$P_{-}:=P_-(\epsilon),\ P_+:=P_+(\epsilon),P_0:=P_-(0)=P_+(0)$.

Notice that in case (a) the middle point of the short arc connecting
points $P_-(\epsilon)$ and $P_+(\epsilon)$ is the point
$P_0=P=(Q+rJ\tau)\in\gamma_{+r}$ and for the case (b) the middle
point is $P_0=(Q-rJ\tau)\in\gamma_{-r}$ (see Fig. 5, 6).

The condition 2. of Definition \ref{def} reads in terms of $F$:
\begin{equation}\label{P}
F(P_-(\epsilon))=F(P_+(\epsilon))
\end{equation}
Differentiating this equality with respect to $\epsilon$ at
$\epsilon=0$  and using the fact that
$\frac{d}{d\epsilon}|_{\epsilon=0}R_\epsilon =J$ we compute in the
case (a):
$$\frac{d}{d\epsilon}|_{\epsilon=0}F(P_-(\epsilon))=dF|_{P_0}(-r\cdot\tau),$$
$$\frac{d}{d\epsilon}|_{\epsilon=0}F(P_+(\epsilon))=dF|_{P_0}(r\cdot\tau).
$$
Here $dF|_{P_0}(w)$ is the differential of the function $F$ at the
point $P_0$ applied to the vector $w$. Thus (\ref{P}) implies that
$$dF|_{P_0}(\tau)=0,$$ where in the last formula $\tau$ should be
understand as the unite tangent vector to $\gamma_{+r}$ at the point
$P_0$ proving the claim in case (a). The case (b) is completely
analogous. This proves that
$$F|_{\gamma_{\pm r}}=const.$$
\end{proof}

\section{\bf Remarkable equation}\label{remarkable}
For any function $F$ which is invariant under $\mathcal{M}$ we can
rewrite equation (\ref{P}) at any non-critical point
$P_0\in\gamma_{\pm r}$ as follows.

Denote by $n=J\tau$ the unite normal vector. From formulas (\ref{a})
we have for the case (a):
\begin{equation}\label{a1}
P_\pm(\epsilon)=Q+rJ(R_{\pm\epsilon}\tau)=Q+rR_{\pm\epsilon}n=
\end{equation}$$
=Q+rn-r(n-R_{\pm\epsilon}n) =P_0-r(I-R_{\pm\epsilon})n.
$$
Analogously for the case (b) we get from (\ref{b}):
\begin{equation}\label{b1}
P_\mp(\epsilon)=Q-rJ(R_{\pm\epsilon}\tau)=Q-rR_{\pm\epsilon}n=
\end{equation}
$$
=Q-rn+r(n-R_{\pm\epsilon}n) =P_0+r(I-R_{\pm\epsilon})n.
$$
Notice that for the unite normal to the curves $\gamma_{+ r}$ and
$\gamma_{- r}$ at $P_0$ one has $n=\pm\frac{\nabla F}{|\nabla F|}$
where the sign is irrelevant since we can change the sign of $F$.
Using this remark we can rewrite the equation (\ref{P}) with the
help of (\ref{a1}),(\ref{b1}) in both cases (a) and (b)
simultaneously:

\begin{equation}\label{rem}
F\left(P_0+r(I-R_{\epsilon})\left(\frac{\nabla F}{|\nabla
F|}\right)(P_0)\right)-
\end{equation}
$$
F\left(P_0+r(I-R_{-\epsilon})\left(\frac{\nabla F}{|\nabla
F|}\right)(P_0)\right)=0, \quad P_0\in \gamma_{\pm r}.
$$
This can be written for $P_0=(x,y)\in\gamma_{\pm r}$ explicitly:

\begin{equation}\label{rem1} F\left(
x+r\frac{F_x(1-\cos\epsilon)+F_y\sin\epsilon}{|\nabla F|};
y+r\frac{F_y(1-\cos\epsilon)-F_x\sin\epsilon}{|\nabla F|} \right)-
\end{equation}
$$F\left(
x+r\frac{F_x(1-\cos\epsilon)-F_y\sin\epsilon}{|\nabla F|};
y+r\frac{F_y(1-\cos\epsilon)+F_x\sin\epsilon}{|\nabla F|} \right)=0.
$$
 The
next step is to expand equation (\ref{rem1}) in power series in
$\epsilon.$ The coefficient at $\epsilon^3$ reads:
\begin{equation}\label{rem2}
(F_{xxx}F_y^3-3F_{xxy}F_y^2F_x+3F_{xyy}F_yF_x^2-F_{yyy}F_x^3)+
\end{equation}
$$
3\beta(F_x^2+F_y^2)^{\frac{1}{2}}(F_{xx}F_xF_y+F_{xy}(F_y^2-F_x^2)-F_{yy}F_xF_y)=0,\quad
(x,y)\in\gamma_{\pm r}.
$$
Remarkably, the left-hand side of (\ref{rem2}) is a complete
derivative along the tangent vector field $v$ to $\gamma_{\pm r}$,
$v=(F_y,-F_x),$ of the following expression which therefore must be
constant:
\begin{equation}\label{rem3}
H(F)+\beta|\nabla F|^3=const, \quad (x,y)\in\gamma_{\pm r},
\end{equation}
where we used the notation
$$
H(F):=F_{xx}F_y^2-2F_{xy}F_xF_y+F_{yy}F_x^2.
$$

Let us remark that the equation (\ref{rem1}) and therefore also
(\ref{rem3}) is valid only for those points, where $\nabla F$ does
not vanish. If the polynomial $F$ is reducible this never happens.
Therefore we proceed as follows. Let us denote by $f_{+r}$
irreducible defining polynomial of $\gamma_{+ r}$, the proof for the
curve $\gamma_{-r}$ is completely the same. Then we have:
$$ F=f_{+r}^k\cdot g,
$$ for some integer $k\geq 1$, and polynomial $g$ not
vanishing on $\gamma_{+ r}$ identically. Given an arc of $\gamma_{+
r}$ where $g$ does not vanish we may assume it is positive on the
arc (otherwise we change the sign of $F$). Moreover, since $f_{+r}$
is irreducible polynomial, then we may assume that $\nabla f_{+r}$
does not vanish on the arc. Therefore the equation (\ref{rem3}) can
be derived in the same manner for the function
$F^{\frac{1}{k}}=f_{+r}\cdot g^{\frac{1}{k}}$ which obviously is
invariant under the map $\mathcal{M}$ exactly as $F$ is.
 Thus we
have
\begin{equation}\label{rem4}
H(f_{+r}\cdot g^{\frac{1}{k}})+\beta|\nabla (f_{+r}\cdot
g^{\frac{1}{k}})|^3=const, \quad (x,y)\in\{f_{+r}=0\}.
\end{equation}
Using the identities which are valid for all $(x,y)\in\{f_{+r}=0\}$
$$
H(f_{+r}\cdot g^{\frac{1}{k}})=g^{\frac{3}{k}}H(f_{+r}),\quad \nabla
(f_{+r}\cdot g^{\frac{1}{k}})=g^{\frac{1}{k}}\nabla (f_{+r}),
$$
we obtain from (\ref{rem4}):
\begin{equation}\label{rem44}
g^{\frac{3}{k}}(H(f_{+r})+\beta|\nabla f_{+r}|^3)=const,
\quad (x,y)\in\gamma_{+r}.
\end{equation}
Raising to the power $k$ back we get:
\begin{equation}\label{rem5}
g^3(H(f_{+r})+\beta|\nabla f_{+r}|^3)^k=const, \quad
(x,y)\in\gamma_{+r}.
\end{equation}
Next we claim the following
\begin{proposition}\label{prop1}
The constant in equation (\ref{rem5})  cannot be 0.
\end{proposition}
\begin{proof}
Recall the formulas for the curvature $k$ of the curve defined
implicitly by $\{f_{+r}=0\}$:
\begin{equation}\label{k}
\textrm{div} \left(\frac{\nabla f_{+r}}{|\nabla f_{+r}|}
\right)=\frac{H(f_{+r})}{|\nabla f_{+r}|^3}=\pm k_{+r}.
\end{equation}

Now we take any point on $\gamma_{+r}$ and substitute into
(\ref{rem5}). This gives that the constant must be non-zero. Indeed,
if the $const$ is zero, then $$ \frac{H(f_{+r})}{|\nabla
f_{+r}|^3}=-\beta.
$$
Then by formulas (\ref{k}) we have
$$
k_{+r}=\pm\beta.
$$
But this is not possible, because we have the bounds on the
curvature of parallel curves (\ref{bounds}).
\end{proof}

\section {\bf Proof of the Main Theorem \ref{main}}\label{mainproof}
In this section we finish the proof of Theorem \ref{main}. We start
with the following:
\begin{theorem} Suppose that magnetic billiard in $\Omega$ admits a
non-constant polynomial integral $\Phi$. If at least one piece of
the boundary is a circular arc, then $\partial\Omega$ is a circle.
\end{theorem}
\begin{proof}
Let us recall that for circular magnetic billiard there exists a
simple integral given by Example \ref{example}. It is very
convenient to pass from $\Phi$ to $F$ defined as above:
$$
F\circ{\mathcal L}=\Phi.
$$
Let us recall that the mapping ${\mathcal L}$ maps any unite vector
$(x,v)$ to the center of the Larmor circle passing through $x$ in
the direction of $v.$ So that if $(Q,v_-)$ is reflected into
$(Q,v_+)$ then the points $P_{\pm}={\mathcal L}(v_{\pm})$ lie on the
circle of radius $r$ which is tangent to the curves $\gamma_{+r}$
and $\gamma_{-r}$ at the points ${\mathcal L}(Q,\tau)$ and
${\mathcal L}(Q,-\tau)$ which are the middle points of the two arcs
of the circle connecting $P_-,P_+$ (see Figures 5,6).

\vskip25mm

\begin{picture}(170,100)(-100,-55)

\qbezier(10,30)(12,88)(70,90)
\qbezier(70,90)(128,88)(130,30)
\qbezier(130,30)(128,-28)(70,-30)
\qbezier(70,-30)(12,-28)(10,30)

\put(70,30){\line(1,2){27}}
\put(70,30){\line(-1,2){27}}
\put(70,30){\line(-1,-2){5}}
\put(70,30){\line(1,-2){5}}

\put(64,19){\circle*{2}}
\put(75,19){\circle*{2}}

\qbezier(57,57)(70,63)(84,57)
\put(57,57){\circle*{2}}
\put(84,57){\circle*{2}}

\put(65,65){\shortstack{$\gamma^{_{(1)}}$}}

\put(43,85){\circle*{2}}
\put(97,85){\circle*{2}}

\put(63,97){\shortstack{$\gamma^{_{(1)}}_{_{-r}}$}}

\put(84,57){\circle*{2}}

\qbezier(84,57)(92,50)(95,34)

\put(96,35){\shortstack{$\gamma^{_{(2)}}$}}

\put(95,34){\circle*{2}}

\qbezier(108,73)(119,64)(124,40)

\put(91,65){\shortstack{$\gamma^{_{(2)}}_{_{-r}}$}}

\put(108,73){\circle*{2}}
\put(124,40){\circle*{2}}

\qbezier(44,43)(46,30)(54,17)
\put(44,43){\circle*{2}}
\put(54,17){\circle*{2}}
\put(40,4){\shortstack{$\gamma^{_{(2)}}_{_{+r}}$}}

\put(70,30){\circle{25}}

\put(63,7){\shortstack{$\gamma^{_{(1)}}_{_{+r}}$}}

\put(128,-3){\shortstack{$C_{-r}$}}

\put(83,20){\shortstack{$C_{+r}$}}

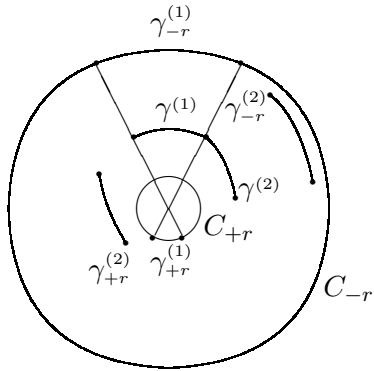
\put(-39,-50){\shortstack{Fig. 7 {\footnotesize Arcs $\gamma^{(1)}, \gamma^{(1)}_{+r}, \gamma^{(1)}_{-r}$ belong to $C, C_{+r}, C_{-r}$.}}}

\end{picture}

Consider two pieces of $\partial\Omega$: $\gamma^{(1)}$  is an arc
of the circle $C$ of radius $d$, and $\gamma^{(2)}$ is the adjacent
piece. Consider also the annulus bounded by the two concentric
circles of $C_{-r}$ of radius $(d+r)$ and $C_{+r}$ of radius $(r-d)$
(see Fig. 7). Let us consider together with the given magnetic
billiard another one acting inside the circle $C$. So the annulus
between the two concentric circles of $C_{-r}$ and $C_{+r}$ is the
phase space of magnetic billiard inside the circle $C$.

We claim that the polynomial function $F$ must have constant value
on every circle concentric with $C$. To show this we denote by
$(\rho, \phi)$ the polar coordinates centered at the center of these
circles. In these coordinates the mapping $P_-\rightarrow P_+$
corresponding to circular billiard reads:
\begin{equation}\label{rotation}
\rho(P_+)=\rho(P_-);\quad \phi(P_+)=\phi(P_-)+ \alpha(\rho),
\end{equation}
where the function $\alpha$ is
\begin{equation}\label{alpha}
\alpha(\rho)=2\arccos \left(\frac{\rho^2+d^2-r^2}{2\rho d}\right).
\end{equation}

It follows from (\ref{alpha}) that this function is analytic in the
annulus between the two circles, $r-d<\rho<d+r$. Consider now the
function
$$\Delta(x,y):=F(P_-)-F(P_+),$$
where $P_-$ has coordinates $(x,y)$ and coordinates of $P_+$ are
determined according to (\ref{rotation}) and (\ref{alpha}). It
follows from the analyticity of function $\alpha$ and polynomiality
of $F$ that the function $\Delta(x,y)$ is analytic on the open
annulus $r-d<\rho<d+r$. Furthermore, since $F$ is built via the
integral $\Phi$ for the billiard inside $\Omega$ and
$\gamma^{(1)}_{\pm r}\subset C_{\pm r}$ then $\Delta$ vanishes on an
open subset of the annulus, and therefore must vanish identically on
the annulus. This fact together with denseness of invariant circles
with irrational rotation numbers for the map (\ref{rotation}),
yields that polynomial $F$ has constant values on every concentric
circle passing inside the annulus, and therefore on every concentric
circle in the plane (not necessarily inside the annulus). This
proves the claim.

Suppose now that the adjacent piece $\gamma^{(2)}$ does not lie on
the circle $C$. This implies then that $\gamma^{(2)}_{+r}$
necessarily intersects an open set of concentric circles. On every
circle $F$ has a constant value by the claim above, and $F$ also is
a constant on $\gamma^{(2)}_{+r}$,  by Proposition \ref{prop}.
Therefore $F$ must be a constant on an open set and hence
everywhere, contrary to the assumptions. This completes the proof.
\end{proof}

Now we are in position to finish the proof of Theorem \ref{main}.
\begin{proof} Consider now the equation (\ref{rem5}) in $\mathbb{C}^2.$ It
follows from (\ref{rem5}) and Proposition \ref{prop1} that the curve
$\{f_{+r}=0\}$ has no singular points in $\mathbb{C}^2,$ since at
singular points both $H(f_{+r})$ and $\nabla (f_{+r})$ vanish.
Moreover, consider now in $\mathbb{C}P^2$ with homogeneous
coordinates $(x:y:z)$ the projective curve $\{\tilde{f}_{+r}=0\}.$
We shall denote homogeneous polynomials corresponding to $f,g$ by
$\tilde{f},\tilde{g}$ respectively. Then the homogeneous version of
(\ref{rem5}) for $(x:y:z)\in \{\tilde{f}_{+r}=0\}$ reads:
\begin{equation}\label{hom}  \tilde{g}^3\left(z\cdot
H(\tilde{f}_{+r})+\beta((\tilde{f}_{+r})_x^2+
(\tilde{f}_{+r})_y^2)^\frac{3}{2}\right)^k=const\cdot z^p.
\end{equation}
Here the power $p=3\deg g+3k(\deg f_{+r}-1)$ must be positive unless
the degree of the polynomial $f_{+r}$ and of $F$ is one. But this is
impossible, due to our convexity assumptions. Let $Z$ be any point
of intersection of $ \{\tilde{f}_{+r}=0\}$ with infinite line
$\{z=0\}.$ Then by (\ref{hom}) for such a point we have two
relations
$$
(\tilde{f}_{+r})_x^2+ (\tilde{f}_{+r})_y^2=0,\quad
x(\tilde{f}_{+r})_x+y(\tilde{f}_{+r})_y+z(\tilde{f}_{+r})_z=x(\tilde{f}_{+r})_x
+y(\tilde{f}_{+r})_y=0.
$$
But these two relations are compatible only in the two cases: either
$$x^2+y^2=z=0,\quad or \quad (\tilde{f}_{+r})_x=
(\tilde{f}_{+r})_y=0.$$ This completes the proof.
\end{proof}

\end{document}